\documentclass[
submission
, nomarks
]{dmtcs-episciences}

\usepackage[utf8]{inputenc}
\usepackage{subfigure}

\newtheorem{myproposition}{Proposition}[section]
\newtheorem{mytheorem}[myproposition]{Theorem}
\newtheorem{mylemma}[myproposition]{Lemma}

\newtheorem{myconjecture}[myproposition]{Conjecture}
\newtheorem{mycorollary}[myproposition]{Corollary}

\newtheorem{myobservation}[myproposition]{Observation}

\usepackage{amsfonts}
\usepackage{amsmath}
\def\gr{\mathcal{G}}
\def\zet{\mathbb{Z}}
\author{Sylwia Cichacz\thanks{This work was partially supported by the Faculty of Applied Mathematics AGH UST statutory tasks within subsidy of Ministry of Science and Higher Education.}
  \and Jakub Przyby{\l}o\thanks{Supported by the National Science Centre, Poland, grant no. 2014/13/B/ST1/01855.}}
\title[Group twin edge coloring of graphs]{Group twin edge coloring of graphs}
\affiliation{
 AGH University of Science and Technology}
\received{2017-9-22}
\revised{2018-2-19,2018-6-5}
\accepted{2018-6-5}
\begin{document}
\publicationdetails{20}{2018}{1}{24}{3848}
\maketitle
\begin{abstract}
  For a given graph $G$, 
the least integer $k\geq 2$ 
such that for every Abelian group $\gr$ of order $k$ there exists a proper edge labeling $f:E(G)\rightarrow \gr$ so that $\sum_{x\in N(u)}f(xu)\neq \sum_{x\in N(v)}f(xv)$ for each edge $uv\in E(G)$ is called
the \textit{group twin chromatic index} of $G$ and denoted by $\chi'_g(G)$.
This graph invariant is related to a few well-known problems in the field of neighbor distinguishing graph colorings. 
We conjecture that $\chi'_g(G)\leq \Delta(G)+3$ for all graphs without isolated edges, where $\Delta(G)$ is the maximum degree of $G$,
and provide an infinite family of connected graph (trees) for which the equality holds. 
We prove that this conjecture is valid 
for all trees, and then apply this result as the base case for proving a general upper bound for all graphs $G$ without isolated edges:  $\chi'_g(G)\leq 2(\Delta(G)+{\rm col}(G))-5$, where ${\rm col}(G)$ denotes the coloring number of $G$.
This improves the best known upper bound known previously only for the case of cyclic groups $\zet_k$.
\end{abstract}
\section{Introduction}
\label{sec:in}

It  is a well-known fact that in any simple graph $G$
there are at least two vertices of the same degree. The situation changes if we consider an edge labeling $f:E(G)\rightarrow \{1,\ldots,s\}$ and calculate the \emph{weighted degree} 
of each vertex $v$ as the sum of labels of all the edges incident with $v$. The labeling $f$ is called \textit{irregular} if the weighted degrees of all the vertices in $G$ are distinct. The least value of $s$ that allows some irregular labeling is called the \textit{irregularity strength of $G$} and denoted by $s(G)$.

The problem of finding $s(G)$ was introduced by Chartrand et al. in \cite{ref_ChaJacLehOelRuiSab1} and investigated by numerous authors  \cite{
ref_AmaTog,Lazebnik,Faudree2,
Frieze,ref_Leh}. 
A tight upper bound $s(G)\leq n-1$, where $n$ is the order of $G$, was proved for all graphs containing no isolated edges and at most
one isolated vertex, except for the graph $K_3$ \cite{ref_AigTri,Nierhoff}.
This was improved for graphs with sufficiently large minimum degree $\delta$ by Kalkowski, Karo\'nski and Pfender \cite{ref_KalKarPfe1}, who proved that $s(G)\leq \lceil 6n/\delta\rceil$, and for graphs with $\delta\geq n^{1/2}\ln n$ 
by Majerski and Przyby{\l}o in \cite{ref_Prz3}, implying that $s(G)\leq (4+o(1))n/\delta+4$ then.

A labeling of the edges of a graph $G$ is called \textit{vertex coloring} if it results in weighted degrees that properly color the vertices (i.e., weighted degrees are required to be distinct only for adjacent vertices). If we use the elements of $\{1,2,\dots,k\}$ to label the edges, such a labeling is called a \textit{vertex coloring $k$-edge labeling}.

The concept of coloring the vertices with the sums of edge labels was introduced for the first time by Karoński, Łuczak and Thomason \cite{ref_KarLucTho}. The authors posed the following question. Given a graph $G$ without isolated edges, 
what is the minimum $k$ such that there exists a vertex coloring $k$-edge labeling? 
We will call this minimum value of $k$ the \textit{sum chromatic number} and denote it by $\chi^\Sigma(G)$. Karoński, Łuczak and Thomason conjectured that $\chi^\Sigma(G)\leq 3$ for every graph $G$ with no isolated edges. 
The first constant bound was proved by Addario-Berry et al. in \cite{ref_AddDalMcDReeTho} ($\chi^\Sigma(G)\leq 30$) and then improved by Addario-Berry et al. in \cite{ref_AddDalRee} ($\chi^\Sigma(G)\leq 16$), Wan and Yu in \cite{ref_WanYu} ($\chi^\Sigma(G)\leq 13$) and finally by Kalkowski, Karo\'nski and Pfender in \cite{ref_KalKarPfe2} ($\chi^\Sigma(G)\leq 5$). Recently Thomassen, Wu and Zhang considered the modulo version of this problem \cite{ref_ThoWuZha}. Specifically, they proved that a non-bipartite $(6k-7)$-edge-connected graph of chromatic number at most $k$ admits a weighting of the edges with labels $1, 2$ such that the resulting weighted degrees reduced modulo $k$ yield a proper vertex coloring of the vertices. A variation of the sum chromatic number
with labels from any Abelian group
is called the \textit{group  sum  chromatic  number} and was studied in \cite{ref_AnhCic1}; more precisely it is the least integer $s$ such that
for 
any Abelian group $\gr$ of order $s$, there exists a function $f\colon E(G)\rightarrow \gr$
which induce a proper coloring of the vertices by their corresponding sums of incident labels.
Such problem was in fact first considered in~\cite{ref_KarLucTho}.


Inspired  by the graph colorings described above,
Andrews et al. \cite{ref_AndHekJohVerPin} turned towards proper edge labelings (with distinct labels on adjacent edges) with the elements of a given $\zet_k$. By a \textit{twin edge coloring} of a graph $G$ (without isolated edges) they denoted a proper edge labeling
$f \colon E(G) \rightarrow \zet_k$ for some $k\geq 2$ such that the induced vertex coloring $w \colon  V (G) \rightarrow \zet_k$ defined by
$w(v) = \sum_{u\in N(v)}f(uv)\pmod k$ is proper. The least integer $k$ admitting such an edge labeling is called the \emph{twin chromatic index} of $G$ and denoted by $\chi'_t(G)$. Note that since $f$ constitutes a proper edge coloring, then $\Delta(G)\leq \chi'(G)\leq \chi'_t(G)$.

Andrews et al. showed that if $G$ is a connected graph of order at least $3$ and size $m$, then $\chi'_t(G)\leq 2^{m-1}$. They also stated the following conjecture and verified it for some classes of graphs:
\begin{myconjecture}[\cite{ref_AndHekJohVerPin}]
If $G$ is a connected graph of order at least $3$ that is not a $5$-cycle, then $\chi'_t(G)\leq \Delta(G)+2$.
\end{myconjecture}
This was a strengthening of a former conjecture of Flandrin et al.~\cite{FlandrinMPSW} with the same thesis but concerning a protoplast of $\chi'_t$ where instead of calculating appropriate sums modulo $k$, we simply compute these in $\zet$. See \cite{BonamyPrzybylo,DongWang_mad,Przybylo_CN_1,Przybylo_CN_2,WangChenWang_planar} for other results concerning this graph invariant.

In \cite{ref_AndJohPin}, Andrews et al. estimated the twin chromatic index for some classes of graphs; in particular they proved the following theorem for trees with
small maximum degree.
\begin{mytheorem}[\cite{ref_AndJohPin}]
If $T$ is a  tree of order at least $3$ and $\Delta(T)\leq6$, then $T$ has a twin edge $(\Delta(T)+2)$-coloring. Moreover if $T$ is a path of order $n \geq 3$, then $\chi'_t(T)= 3$. \label{drzewa}
\end{mytheorem}
For an integer $r\geq 2$, a tree $T$ is called $r$-\textit{regular} if each non-leaf of $T$ has degree $r$.

\begin{mytheorem}[\cite{ref_Jon}]\label{regular}
If $T$ is a regular tree of order at least $6$, then
 $\chi'_t(T)\leq \Delta(T)+2$.
Moreover if $\Delta(T)\equiv 1 \pmod 4$ then $\chi'_t(T)=\Delta(T)+2$.
\end{mytheorem}

As for a general upper bound, the best thus far result is due to Johnston \cite{ref_Jon}, who proved the following.

\begin{mytheorem}[\cite{ref_Jon}]\label{JohnstonTh}
If $G$ is a connected graph of order at least $3$, then $\chi'_t(G)\leq 4\Delta(G)-3$.
\end{mytheorem}

Assume $\gr$ is an Abelian group  of order $k$ with the operation denoted by $+$ and the identity element $0$. For convenience we will write $ma$ to denote $a+a+\ldots+a$ (where element $a$ appears $m$ times), $-a$ to denote the inverse of $a$ and we will use $a-b$ instead of $a+(-b)$. Moreover, the notation $\sum_{a\in S}{a}$ will be used as a short form for $a_1+a_2+a_3+\ldots$, where $a_1, a_2, a_3, \ldots$ are all the elements of the set $S$.
We will call a proper edge labeling $f:E(G)\rightarrow \gr$ a $\gr$-\textit{twin edge coloring} if the resulting \emph{weighted degrees}, defined for every vertex $v\in V(G)$ as the sum
(in $\gr$):
$$
w(v)=\sum_{u\in N(v)}f(uv),
$$
yield a proper vertex coloring of $G$, i.e. we have $w(u)\neq w(v)$ for every edge $uv\in E(G)$.
We will also call $w(v)$ the \emph{color} of a vertex $v$ or the \emph{sum at} $v$, while such a labeling $f$ will be referred to as \emph{neighbor sum distinguishing} as well.
Generalizing the concept of the twin chromatic index,
the least integer $k\geq 2$ for which $G$ has a $\gr$-twin edge coloring for every Abelian group $\gr$ of order $k$ is called the \textit{group twin chromatic index} of $G$ and is denoted by $\chi'_g(G)$. Obviously $\chi'_t(G)\leq\chi'_g(G)$ for any graph $G$ (without isolated edges), and there are plenty of graphs for which $\chi'_t(G)<\chi'_g(G)$ (cf. Theorem~\ref{regular} and Observation~\ref{sharp}).
Note here also that the fact that $\chi'_g(G)\leq K$ for a given graph $G$ and a constant $K$ does not guarantee that for every
group $\gr$ of order $k>K$ there exists a $\gr$-twin edge coloring  of  $G$,  
and see our concluding Section~\ref{SectionOurConcludingRemarks} for a further discussion concerning this issue.

Surprisingly, in this paper we in fact provide an infinite family of connected graphs (which are trees) for which $\chi'_g(G)\geq \Delta(G)+3$, see Theorem~\ref{sharp}. Such phenomenon is not known for a few forefathers of this graph invariant discussed above
(cf. additionally the conjecture in~\cite{Zhang}, and the best known result concerning this from~\cite{Hatami}), for which $\Delta(G)+2$ labels are suspected to suffice for almost all connected graphs. In case of the group twin chromatic index, we conjecture that $\Delta(G)+3$ labels should always be sufficient and confirm this for all trees (which are not isolated edges). 
On the way we also discuss many rich families of trees for which such an upper bound can be improved.
We then use our result concerning trees as a base case in a proof of a general upper bound for all graphs for which the parameter $\chi'_g(G)$ is well defined.
Namely, by means of a straightforward algorithmic construction (efficient for all connected graphs except possibly some trees) 
we finally provide a two-fold improvement of Theorem~\ref{JohnstonTh} of Johnston, whose proof is rather complex and lengthy. That is, we show that $\chi'_g(G)\leq 2(\Delta(G)+{\rm col}(G))-5$ for every 
graph $G$ without isolated edges, 
where ${\rm col}(G)$ denotes the coloring number of $G$ (which is equal to the degeneracy of $G$ plus $1$).
This strengthens the thesis of Theorem~\ref{JohnstonTh}, as ${\rm col}(G)-1\leq \Delta(G)$, while this inequality is sharp for many graph classes (e.g. for planar graphs, for which ${\rm col}(G)\leq 6$ whereas $\Delta(G)$ is unbounded), and extends it towards colorings with elements of all Abelian groups, not just $\mathbb{Z}_k$.

\section{Preliminaries}
Assume $\gr$ is an Abelian group of order $n$. The order of an element $a\neq 0$ is the smallest $r$ such that $ra=0$. Recall that any group element $\iota\in\gr$ of order 2 (i.e., $\iota\neq 0$ such that $2\iota=0$) is called \emph{involution}.  It is well known by Lagrange Theorem that the order of any element of $\gr$ divides $|\gr|$ \cite{ref_Gal}. Therefore every  group of odd order has no involution. The fundamental theorem of finite Abelian groups states that a finite Abelian group $\gr$ of order $n$ can be expressed as the direct product of cyclic subgroups of prime-power orders. This implies that
$$\gr\cong\zet_{p_1^{\alpha_1}}\times\zet_{p_2^{\alpha_2}}\times\ldots\times\zet_{p_k^{\alpha_k}}\;\;\; \mathrm{where}\;\;\; n = p_1^{\alpha_1}\cdot p_2^{\alpha_2}\cdot\ldots\cdot p_k^{\alpha_k}$$
and $p_1,p_2,\ldots,p_k$ are not necessarily distinct primes. This product is unique up to the order of the direct product. When $t$ is the number of these cyclic components
whose order is a multiple of $2$, then $\gr$ has $2^t-1$ involutions. In particular every cyclic group of even order has exactly one involution. 

The sum of all the
group elements is equal to the sum of the involutions and the neutral element. The following lemma was proved in~\cite{CN} (Lemma 8).

\begin{mylemma}[\cite{CN}]\label{involution} Let $\gr$  be an Abelian group.
\begin{enumerate}
	\item If $\gr$ has exactly one involution $\iota$, then $\sum_{g\in \gr}g= \iota$.
	\item If $\gr$ has no involution, or more than one involution, then $\sum_{g\in \gr}g=0$.
\end{enumerate}
\end{mylemma}

Anholcer and Cichacz proved a lemma about a partition of the set of all elements of $\gr$ of order at most $2$ into two zero-sum sets  (see \cite{ref_AnhCic1}, Lemma 2.4). Their result along with results proved by Cichacz (see \cite{ref_Cic}, Lemma 3.1) give the  following lemma.
\begin{mylemma}[\cite{ref_AnhCic1,ref_Cic}]\label{involutions} Let $n_1$, $n_2$, $n_3$ be non-negative integers such that $n_1+n_2+n_3= 2^k$ with integer $k \geq2$, and $k>2$ if $n_1n_2n_3\neq0$. Let $\gr$  be an Abelian group  with involution set $I^* =\{\iota_1,\iota_2,\ldots,$ $\iota_{2^k-1}\}$ and set $I = I^*\cup\{0\}$.
Then there exists a partition $A=\{A_1, A_2, A_3\}$ of $I$ such that
\begin{enumerate}
  \item $n_1= |A_1|$, $n_2= |A_2|$, $n_3= |A_3|$,
  \item $\sum_{a\in A_i}a=0$ for $i\in\{1,2,3\}$,
\end{enumerate}
 if and only if $n_1,n_2,n_3\not\in\{2,2^k-2\}$.
\end{mylemma}

\section{Group twin edge coloring for trees}
We start with the following lemmas:
\begin{mylemma} Let $\gr$ be a finite Abelian group of  order $|\gr|\geq6$ having at most one involution.
For any elements $a,b\in\gr$ such that $2b=0$, $a\neq b$ there exist  elements $x,y$, $x\neq y$, $\{a,b\}\cap \{x,y\}=\emptyset$ such that   $x+y=b-a$.\label{trzyzero} 
\end{mylemma}

\begin{proof}
 For $z\in \gr$ let $S_{z/2}=\{t \in \gr: 2t =z \}$. Observe that if there exist pairwise distinct $t_1,t_2,t_3 \in S_{z/2}$, then $t_1-t_2$ and $t_1-t_3$ are distinct involutions, therefore $|S_{z/2}|\leq 2$ for any $z \in \gr$.

We will show that there exists a desired solution of the equation $x+y+a-b=0$. Suppose first that $0 \not \in \{a,b\}$. Then we may set $x = b-a$ and $y = 0$ 
unless $b = 2a$. In the latter case we must however have $|\gr|>6$ (as $b$ is the only involution in $\gr$) and for any $c \in \gr \setminus(S_{-a/2}\cup \{0, b, a,-a\})\neq\emptyset$,  $x = a+c$
and $y = -c$ yield a solution. Similarly,   if $b
\neq 0$ and $a=0$,
then for  $c \in \gr \setminus (S_{b/2}\cup \{0, b\})$  by setting $x = b + c$ and $y =- c$ we obtain a desired solution. Finally, if $b = 0$, then $x = c-a$ and
$y =-c$ are valid provided $c \in \gr \setminus(S_{a/2}\cup \{0, a,-a,2a\})$, while to see that this last set is always nonempty it suffices to note that if $|\gr|=6$ and $|S_{a/2}|=2$, then $-a=2a$.
\end{proof}

We also prove a somewhat stronger version of Lemma~\ref{involutions} (for the case when $n_3=0$).
\begin{mylemma}\label{involutions2} Let $\gr$  be an Abelian group  with involution set 
$I^* =\{\iota_1,\iota_2,\ldots,\iota_{2^k-1}\}$, $k \geq2$, and let $I = I^*\cup\{0\}$.
Given an element $\iota\in I^*$ and positive integers $n_1$, $n_2$  such that $n_1+n_2= 2^k$, $n_1\neq 2$ and $n_2 \geq 3$, there exists a partition $A=\{A_1, A_2\}$ of $I$ such that
\begin{enumerate}
  \item $n_1= |A_1|$, $n_2= |A_2|$,
  \item $\sum_{a\in A_i}a=0$ for $i=1,2$,
	\item $\iota \not \in A_1$.
\end{enumerate}

\end{mylemma}

\begin{proof}
  Recall that since $I=\{0,\iota_1,\ldots,\iota_{2^k-1}\}$ is a subgroup of $\gr$, we have $I\cong (\zet_2)^k$. Observe that $n_1$ and $n_2$ have the same parity. If they are both even then there exists a partition $A=\{A_1, A_2\}$ of $I$ such that
$n_1= |A_1|$, $n_2= |A_2|$ and $\sum_{a\in A_i}a=0$ for $i=1,2$ by Lemma~\ref{involutions}. If now $\iota \in A_2$, we are done. If $\iota \in A_1$, then for some $\iota'\in A_2$  there exists exactly one  $x\in I^*$ such that $\iota'+x=\iota$. Define now $A'_1=\{a+x,a\in A_1\}$ and $A'_2=\{a+x,a\in A_2\}$. Note that $\sum_{a\in A'_i}a=\sum_{a\in A_i}a +|A_i|x=0$ for $i=1,2$ because $|A_i|$ is even. Hence  $\{A'_1,A_2'\}$ forms the desired partition.\\

Assume now that $n_1$ and $n_2$ are both odd.   One can see that the lemma holds for $k=2$.  Suppose that the statement of the theorem is true for all groups with at least three and less than $2^k-1$ involutions. Let us establish it for groups with $2^k-1$ involutions.  Let  $n_1'=n_1\pmod {2^{k-1}}$, $n_2'=n_2\pmod {2^{k-1}}$. Let $\iota=(i_1,i_2,\ldots,i_{k})$ and set $\iota'=(i_1,i_2,\ldots,i_{k-1})$.  If $n_2>2^{k-1}$, then note that $n_1= n_1'$,  $n_2 = n_2'+ 2^{k- 1}$ and there exists a partition $A'=\{A_1', A_2'\}$ of $(\zet_2)^{k-1}$ such that
$n_1'= |A_1|'$, $n_2'= |A_2'|$, $\sum_{a\in A_i'}a=0$ for $i=1,2$ by Lemma~\ref{involutions}. If now $i_{k}=1$, then replace each element $(y_1,y_2,\ldots,y_{k-1})$ of  $(\zet_2)^{k-1}$ in any $A'_i$ by the element $(y_1,y_2,\ldots,y_{k-1},0)$ of  $(\zet_2)^{k}$. Define $A_2''=\{(y,1),y\in (\zet_2)^{k-1}\}$ and set $A_1=A'_1$, $A_2=A_2'\cup A_2''$. Suppose now $i_{k}=0$. Note that then $\iota'\neq (0,0,\ldots,0)$. Replace each element $(y_1,y_2,\ldots,y_{k-1})$ of  $(\zet_2)^{k-1}$ in any $A'_i$ by the element $(y_1,y_2,\ldots,y_{k-1},0)$ except for $\iota'$ and some other element $x'$ such that $x'$ belongs to the same partition set as $\iota'$ (note we may do so, 
since $\iota'\neq 0$); for them we put $(\iota',1)$ and $(x',1)$. Define $A_2''=\{(y,1), y\in (\zet_2)^{k-1},y\notin\{\iota',x'\}\}$ and set $A_1=A_1'$ and $A_2=A_2'\cup A_2''\cup\{(\iota',0),(x',0)\}$. For $n_2<2^{k}$ there exists a partition $A'=\{A_1', A_2'\}$ of $(\zet_2)^{k-1}$ such that $\iota'\in A_2'$,
$n_1'= |A_1|'$, $n_2'= |A_2'|$, $\sum_{a\in A_i'}a=0$ for $i=1,2$ by the induction hypothesis. We then define $A_1$ and $A_2$ analogously as above.
\end{proof}
\begin{mylemma} Let $T$ be a tree of order at least $3$ with   maximum degree $t$  and $\gr$ be an Abelian group  of odd order $k\geq\max\{7, t+2\}$.
Then there exists a $\gr$-twin edge coloring  of $T$,
unless $T$ is a $(3^p-2)$-regular tree and $\gr\cong(\zet_3)^p$ for some integer $p$.\label{treeodd}
\end{mylemma}

\begin{proof}Let $T$ be a tree with  maximum degree $t$  and $\gr$ be an Abelian group  of odd order $k\geq\max\{7, t+2\}$
such that either $T$ is not a $t$-regular tree or $\gr\not\cong(\zet_3)^p$ with $3^p=t+2$.
Assume $T$ is rooted at a vertex $v_0$ with minimum degree $t'\geq 2$ in $T$.
Let
$v_1,v_2,\ldots,v_m$ be all the remaining  vertices of $T$ which are not leaves, and denote their corresponding numbers of children by $r_1,r_2,\ldots,r_m$. Note that $r_i$=deg$(v_i)-1$. Let $N_i=\{v_iw: w$ is a child of $v_i\}$ for $i=0,1,\ldots,m$.\\

If $t'$ is even, take $t'/2$ pairs $(d_1,-d_1),$ $\ldots,$ $(d_{t'/2},-d_{t'/2})$ of distinct elements of $\gr$ and arbitrarily label all edges incident with $v_0$ with these. Then $w(v_0)=0$.

Assume now $t'$ is odd. Suppose first that $t'<k-2$. Then in fact $t'\leq k-4$ and there exist nonzero pairwise distinct elements  $x,y,a\in\gr$ such that $x+y+a=0$ by Lemma~\ref{trzyzero}. Take $x,y,a$ and $(t'-3)/2$ distinct pairs $(d_1,-d_1),$ $\ldots,$ $(d_{(t'-3)/2},-d_{(t'-3)/2})$ from the remaining elements of $\gr$ and arbitrarily label all edges incident with $v_0$ with these. Observe that $w(v_0)=0$ then.
Thus suppose now that $t'=k-2$. Then we must have $t'=t$ and by our choice of $v_0$, $T$ must be a $(k-2)$-regular tree.
Therefore $\gr\not\cong(\zet_3)^p$ for all integers $p$, and hence there exist two distinct (nonzero) elements $a$ and $b$  such that $a+b=-a$.
Then take $0,a,b$ and $(t-3)/2$ pairs $(d_1,-d_1),$ $\ldots,$ $(d_{(t-3)/2},-d_{(t-3)/2})$ from the remaining elements of $\gr$ and  label all edges incident with $v_0$ with them. Observe that $w(v_0)=a+b=-a$ then.

Observe that  for any edge $e\in N_0$ 
we have $f(e)\neq w(v_0)$.
In each next step now we will label edges from the set $N_i$ only if the edge between $v_i$ and its parent $v^i$ is already labeled. We will do it in such a way that $f(e)\notin \{f(v^iv_i),0\}$ for any $e\in N_i$ and  $\sum_{e\in N_i}f(e)=0$ if $r_i>1$. Note that
then
any vertex $v\neq v_0$ with deg$(v)\neq2$ will have assigned a color equal to the label of the edge between $v$ and its parent.\\

Suppose first that $r_i$ is even. Then $r_i+3\leq k$ and one can easily see that we can pick $r_i/2$  pairs of all distinct elements  $(d_1,-d_1),$ $\ldots,$ $(d_{r_i/2},-d_{r_i/2})$ from $\gr$ not including $f(v^iv_i)$.
We thus label the edges of the set $N_i$ with these and we are done.

Assume now $r_i$ is odd. Then $r_i+4\leq k$. Recall that $f(v^iv_i)\neq 0$. For $r_i=1$ take any nonzero element $g$ from $\gr$ such that $g\notin\{f(v^iv_i),w(v^i)-f(v^i v_i)\}$ (we can do this because $|\gr|\geq 7$) and label the edge from $N_i$. Now, if $r_i\geq 3$, by Lemma~\ref{trzyzero} there exist two distinct nonzero elements $x$ and $y$ such that $-f(v^iv_i)\notin \{x,y\}$ and $x+y=-f(v^iv_i) $. Take $-f(v^iv_i),x,y$ and $(r_i-4)/2$ pairs $(d_1,-d_1),$ $\ldots,$ $(d_{(r_i-4)/2},-d_{(r_i-4)/2})$ from the remaining 
elements  of $\gr$ and arbitrarily label all edges from $N_i$ with these.
\end{proof}

A $\gr$-twin edge coloring of a graph $G$ is said to be \textit{nowhere-zero} if it uses no label $0$ on any edge of $G$.
Observe that by the proof above, if a tree $T$ has even maximum degree at least $6$, then $T$ has a nowhere-zero $\gr$-twin edge coloring for any Abelian group $\gr$ of odd order $|\gr|\geq\Delta(T)+3$. Moreover, if $T$ has odd maximum degree at least $5$, then it has a nowhere-zero $\gr$-twin edge coloring for any Abelian group $\gr$ of odd order $|\gr|\geq\Delta(T)+2$, except for the case when $T$ is a $(|\gr|-2)$-regular tree.

Thus we know that $\chi'_g(T)\leq \Delta(T)+3$ for all trees $T$ of maximum degree at least $5$ except the regular trees (for which $\chi'_g(T)\leq \Delta(T)+4$ holds). Before we show that such an upper bound is true for all trees, we present an infinite family of (regular) trees witnessing
that this bound cannot be in general improved.

\begin{myobservation}\label{sharp}
If $T$ is a regular tree of order at least $7$ such that $\Delta(T)=3^{2p+1}-2$ for some positive integer $p$,
then $\chi'_g(T)\geq \Delta(T)+3$.
\end{myobservation}
\begin{proof} Observe that $\Delta(T)=3^{2p+1}-2\equiv 1 \pmod 4$. Therefore $\chi'_g(T)\geq \Delta(T)+2$ by Theorem~\ref{regular}.
Assume, to the contrary, that $\chi'_g(T)= \Delta(T)+2$. Let $f$ be a $\gr$-twin edge coloring  of $T$ with $\gr=(\zet_3)^{2p+1}$,
and let $v_0$ be an internal vertex of $T$, i.e.
such that deg$(v_0)=\Delta(T)$. Then there are
exactly two elements $a,b\in(\zet_3)^{2p+1}$ that are not assigned to any edge incident with $v_0$.  Since $w(v_0)=-(a+b)$ and it is impossible in  $(\zet_3)^{2p+1}$ that $-(a+b)=a$ or $-(a+b)=b$, we obtain that $w(v_0) = c_1 = f(v_0v_1)$ for some $v_1 \in  N (v_0)$, $c_1\notin\{a,b\}$.  If $v_1$ is a leaf of $T$, then $w(v_1)=w(v_0)$, a contradiction. Thus $v_1$ is not a leaf of $T$ and so deg$(v_1) =\Delta(T)$. Analogously as above, $w(v_1)=c_2=f(v_1v_2)$ for some $v_2 \in  N (v_1)$, and since $f$ distinguishes $v_1$ and $v_2$ by their corresponding sums, then $c_2\neq c_1$, and hence $v_2\neq v_0$. If $v_2$ is a leaf of $T$, we obtain a contradiction. Otherwise we continue this process, and since $T$ has no cycles and is finite, we eventually must 
reach
a leaf $v_r$ such that $v_r\in N(v_{r-1})$ and $w(v_r)=c_{r}=w(v_{r-1})$, a contradiction.
\end{proof}

\begin{mylemma} Let $T$ be a tree with   maximum degree $t\geq5$ such that  any vertex of degree $2$ has at least one neighbor of degree $2$ and let $\gr$ be an Abelian group with exactly one involution $\iota$,  $|\gr|\geq t+2$. Then there exists a $\gr$-twin edge coloring  of $T$.\label{inwolucja}
\end{mylemma}
\begin{proof}~
Set $k=|\gr|$.
We will define a $\gr$-{twin edge coloring} $f\colon E(T)\rightarrow \gr$.


Let $T$ be  rooted at a vertex $v_0$ of degree $t$. 
Let $v_1,v_2,\ldots,v_m$ be all the remaining  vertices of $T$ which are not leaves, and denote their corresponding numbers of children by $r_1,r_2,\ldots,r_m$.  Set $N_i=\{v_iw: w$ is a child of $v_i\}$ for $i=0,1,\ldots,m$.\\

If $t$ is even take $t/2$ distinct pairs $(d_1,-d_1),$ $\ldots,$ $(d_{t/2},-d_{t/2})$ of the  elements of $\gr$ and arbitrarily label all edges incident with $v_0$ with them. Then $w(v_0)=0$. If $t$ is odd, then $k>t+2\geq7$. Therefore, by Lemma~\ref{trzyzero}, in the group $\gr$ we have nonzero elements $a\neq b$ such that $a+b=\iota$.
Take $0,a,b$ and $(t-3)/2$ distinct pairs $(d_1,-d_1),$ $\ldots,$ $(d_{(t-3)/2},-d_{(t-3)/2})$ from the remaining elements of $\gr$ and  label all edges incident with $v_0$ with them. Observe that $v_0$ is assigned a color $w(v_0)=a+b=\iota$ then.

The main idea of the proof is similar to the proof of Lemma~\ref{treeodd} --
in each next step we will label edges from the set $N_i$ only if the edge between $v_i$ and its parent (say $v^i$) is already labeled.
This time however the label $0$ is allowed  for an edge, but only for edges belonging to a set $N_i$ with $r_i>1$. \\

Suppose first that $r_i$ is even. Then $r_i+4\leq k$ and one can easily see that we can pick $r_i/2$ distinct pairs  $(d_1,-d_1),$ $\ldots,$ $(d_{r_i/2},-d_{r_i/2})$ of  elements of $\gr$ not including $f(v^iv_i)$. We label  the edges in $N_i$ with these and we are done.

Assume now $r_i$ is odd, thus $r_i+3\leq k$.  For $r_i=1$, take any nonzero element $g$ from $\gr$ such that $g\notin\{f(v^iv_i),w(v^i)-f(v^iv_i)\}$ (we can do this because $|\gr|\geq 7$) and label the edge from $N_i$. Assume now that $r_i\geq 3$. If $f(v^iv_i)\neq 0$, take $0$ and $(r_i-1)/2$ distinct pairs $(d_1,-d_1),$ $\ldots,$ $(d_{(r_i-1)/2},-d_{(r_i-1)/2})$ from  elements of $\gr$ that are different from $f(v^iv_i),-f(v^iv_i)$ and arbitrarily label  the edges of the set $N_i$ with these elements. Suppose then that $f(v^iv_i)=0$. By Lemma~\ref{trzyzero}, there exist distinct nonzero elements $x,y\in\gr$ such that $x+y=\iota$. 
Thus take $\iota,x,y$ and $(r_i-3)/2$ distinct pairs $(d_1,-d_1),$ $\ldots,$ $(d_{(r_i-3)/2},-d_{(r_i-3)/2})$ from the remaining  
elements of $\gr$ and  label all edges from $N_i$ with them.
\end{proof}


The famous Catalan-Mih\v{a}ilescu Theorem says  that the only solution in the natural numbers of the equation $x^a - y^b = 1$ for $a, b > 1$, $x, y > 0$ is $x = 3$, $a = 2$, $y = 2$, $b = 3$ \cite{Mih}. Therefore $\chi'_g(K_{1,2^{p}-3})\leq2^{p}-1$ for $p\geq3$ by Lemma~\ref{treeodd}.
However the tree $K_{1,2^{p}-3}$ does not have a $(\zet_2)^p$-twin edge coloring.
Indeed, for suppose 
we are able to label $K_{1,2^{p}-3}$ appropriately with elements from $(\zet_2)^p$.
In such a situation we
would have to use $2^{p}-3$ distinct elements of $(\zet_2)^p$ on the edges, which would leave us
three distinct elements, $g_1,g_2,g_3$ unused. The weighted degree of the central vertex would
be $-(g_1+g_2+g_3)$. This should be distinct from all other weighted degrees, so one of the
equalities $-(g_1+g_2+g_3)=g_1$, $-(g_1+g_2+g_3)=g_2$ or  $-(g_1+g_2+g_3)=g_3$ would have to be satisfied. In all cases it
follows that $g_i=g_j$ for  $i\neq j$, $i,j\in\{1,2,3\}$, a contradiction. Moreover we could extend this arguments similarly as in the proof of  Observation~\ref{sharp} to the case of any $(2^{p}-3)$-regular tree. However, for a group $\gr$ having more than one involution we are able to prove the following.

\begin{mylemma} Let $T$ be a tree with   maximum degree  $t\geq5$ such that  any vertex of degree $2$ has at least one neighbor of degree $2$ and let $\gr$ be an Abelian group  of order $k\geq t+2$ with more than one involution. Let $p$ be an integer such that $p\in\{\log_2(t+2),\log_2(t+3)\}$. If  $\gr \not \cong (\zet_2)^p$, then there exists a $\gr$-twin edge coloring  of $T$.\label{inwolucje}
\end{mylemma}
\begin{proof}
Let $\gr$ be an Abelian group of order $k$ with involution set $I^* =\{\iota_1,\iota_2,\ldots,\iota_{2^p-1}\}$, $p >1$. Let $I=I^* \cup\{0\}$. Note that $\gr$ has  even order. Since the group $\gr$  can be expressed as the direct product of cyclic subgroups of prime-power orders one can easily see that either $k=2^p$ or $k>2^{p+1}-1$.  We will define a $\gr$-{twin edge coloring} $f\colon E(T)\rightarrow \gr$.

As before let $T$ be a rooted tree with  root $v_0$ such that $\deg(v_0)=\Delta(T)=t$. 
Let $v_1,v_2,\ldots,v_m$ be all the remaining  vertices of $T$ which are not leaves, and denote their corresponding numbers of children by $r_1,r_2,\ldots,r_m$. Let $N_i=\{v_iw: w$ is a child of $v_i\}$ for $i=0,1,\ldots,m$.\\

If $t$ is even and $k= t+2$, then  from the assumption that $\gr \not \cong (\zet_2)^p$ we deduce that $t>2^{p+1}-3$. Take $\iota_2,\iota_3,\ldots,\iota_{2^p-1}$ and $(t+2-2^p)/2$ distinct pairs $(d_1,-d_1),$ $\ldots,$ $(d_{(t+2-2^p)/2},-d_{(t+2-2^p)/2})$ from  elements of $\gr$ and arbitrarily label all edges incident with $v_0$ with these. Observe that then $w(v_0)=\iota_1$ by Lemma~\ref{involution}.
Suppose $t$ is even and $k\geq t+4$. If now $t<2^p-2$  then there exists a partition $A=\{A_1, A_2, A_3\}$ of $I$ such that $|A_1|=t\geq6$, $|A_2|=1$, $|A_3|=2^p-1-t$ and $\sum_{a\in A_i}a=0$ for $i\in\{1,2,3\}$ by Lemma~\ref{involutions}.
Note that $0\notin A_1$. Label all edges incident with $v_0$ with the elements of $A_1$, hence $w(v_0)=0$.
If now $t\geq 2^p-2$ then obviously $k>2^{p+1}-1$. Thus by Lemma~\ref{involutions}, there exists a partition $A=\{A_1, A_2, A_3\}$ of $I$ such that $|A_1|=2^{p}-4$, $|A_2|=1$, $ |A_3|=3$ and $\sum_{a\in A_i}a=0$ for $i\in\{1,2,3\}$. We take elements from $A_1$ and $(t+4-2^p)/2$ distinct pairs $(d_1,-d_1),$ $\ldots,$ $(d_{(t+4-2^p)/2},-d_{(t+4-2^p)/2})$ from the  elements of $\gr$ and label all edges incident with $v_0$ with these. Observe that then $w(v_0)=0$.

Assume now that $t$ is odd. If $t\geq2^{p}-1$
, then take $\iota_1,\iota_2,\ldots,\iota_{2^p-1}$ and $(t+1-2^p)/2$ distinct pairs $(d_1,-d_1),$ $\ldots,$ $(d_{(t+1-2^p)/2},-d_{(t+1-2^p)/2})$ from the  elements of $\gr$ and arbitrarily label all edges incident with $v_0$ with them. Observe that then $w(v_0)=0$.
If  $t<2^p-3$   then there exists a partition $A=\{A_1, A_2,A_3\}$ of $I$ such that $|A_1|=t\geq 5$, $|A_2|=1$,  $|A_3|=2^p-1-t\geq4$ and $\sum_{a\in A_i}a=0$ for $i\in\{1,2,3\}$ by Lemma~\ref{involutions}.
Label all edges incident with $v_0$ by the elements of $A_1$, thus $w(v_0)=0$. Finally suppose that $2^p-3= t$. By the assumption $\gr \not \cong (\zet_2)^p$, this implies that $k>2^p$. Thus  by Lemma~\ref{involutions} there exists a partition $A=\{A_1, A_2,A_3\}$ of $I$ such that $|A_1|=t-2\geq3$, $|A_2|=1$,  $|A_3|=2^p+1-t$ and $\sum_{a\in A_i}a=0$ for $i\in\{1,2,3\}$. Take elements from $A_1$ and one pair $(d_1,-d_1)$  from the  elements of $\gr$ and  label all edges incident with $v_0$ with these elements. Observe that then $w(v_0)=0$.

The main idea of the further part of the proof is the same as in the proof of Lemma~\ref{inwolucja}.
In each next step we will label edges from the set $N_i$ only if the edge between $v_i$ and its parent (say $v^i$) is already labeled.

Suppose first that $r_i$ is odd, thus $r_i+3\leq k$. If $r_i=1$ then we are taking any non-zero element $g$ from $\gr$ such that $g \notin \{f(v^iv_i),w(v^i)-f(v^iv_i)\}$ (we can do this because $|\gr|\geq 7$) and label the edge in $N_i$ with it. Let $r_i\geq3$. Suppose that $f(v^iv_i)\notin I^*$. Then  for $r_i<2^p-3$  there exists a partition $A=\{A_1, A_2,A_3\}$ of $I$ such that $|A_1|=r_i$, $|A_2|=1$,  $|A_3|=2^p-1-r_i$ and $\sum_{a\in A_j}a=0$ for $j\in\{1,2,3\}$ by Lemma~\ref{involutions}. Label all edges from $N_i$ by the elements of $A_1$. If $2^p-3= r_i$, then $t\geq k+2$ and $\gr\not \cong (\zet_2)^p$ imply that $k>2^p$ and $r_i\geq5$. Therefore  there exists a partition $A=\{A_1, A_2,A_3\}$ of $I$ such that $|A_1|=r_i-2$, $|A_2|=1$,  $|A_3|=2^p+1-r_i$ and $\sum_{a\in A_j}a=0$ for $j\in\{1,2,3\}$ by Lemma~\ref{involutions}. Take elements from $A_1$ and one pair $(d_1,-d_1)$ such that $f(v^iv_i)\notin \{d_1,-d_1\}$  from the  elements of $\gr$ and  label all edges from $N_i$ with these. For $r_i\geq2^p-1$ take $\iota_1,\iota_2,\ldots,\iota_{2^p-1}$ and $(r_i+1-2^p)/2$ pairs $(d_1,-d_1),$ $\ldots,$ $(d_{(r_i+1-2^p)/2},-d_{(r_i+1-2^p)/2})$ from the  elements of $\gr$  that are different from $f(v^iv_i)$ and  label  the edges of the set $N_i$ with them. Assume now $f(v^iv_i)\in I^*$. Then  for $r_i<2^p-1$    there exists a partition $A=\{A_1, A_2\}$ of $I$ such that $|A_1|=r_i$,   $|A_2|=2^p-r_i$,  $\sum_{a\in A_j}a=0$ for $j\in\{1,2\}$ and $f(v^iv_i)\notin A_1$  by Lemma~\ref{involutions2}. Label all edges from $N_i$ by the elements of $A_1$ then. For $r_i\geq2^p-1$   there exists a partition $A=\{A_1, A_2\}$ of $I$ such that $|A_1|=2^p-3$,   $|A_2|=3$,  $\sum_{a\in A_j}a=0$ for $j\in\{1,2\}$ and $f(v^iv_i)\notin A_1$  by Lemma~\ref{involutions2}. Recall that for $2^p-1\leq r_i$ we have $k>2^{p+1}-1$.  Label all edges from $N_i$ by the elements of $A_1$ and $(r_i+3-2^p)/2$ distinct pairs $(d_1,-d_1),$ $\ldots,$ $(d_{(r_i+3-2^p)/2},-d_{(r_i+3-2^p)/2})$ from the  elements of $\gr$.

Assume that $r_i$ is even, hence $k\geq r_i+4$.
If $r_i=2$ and $\gr \not\cong (\zet_2)^p$, take any $(d_1,-d_1)$ such that $f(v^iv_i)\notin \{d_1,-d_1\}$ to label the two edges in $N_i$ and we are done. For   $\gr \cong (\zet_2)^p$ one can see that we are able to take $g_1,g_2\in \gr$ such that $g_1\neq g_2$, $f(v^iv_i)\not\in\{g_1,g_2\}$ and $g_1+g_2\neq w(v^i)-f(v^iv_i)$. Label the edges from $N_i$ by $g_1$ and $g_2$. Let $r_i\geq4$.   
Assume first that $r_i\geq 2^p$. If $f(v^iv_i)\notin I$, take all elements of $I$
and $(r_i-2^p)/2$ distinct pairs $(d_1,-d_1),$ $\ldots,$ $(d_{(r_i-2^p)/2},-d_{(r_i-2^p)/2})$ from the  elements of $\gr$ not including $f(v^iv_i)$ and  label all edges from $N_i$ with them.
If on the other hand $f(v^iv_i)\in I$ then there exists a partition $A=\{A_1, A_2\}$ of $I$ such that $|A_1|=2^{p}-4$,   $|A_2|=4$,  $\sum_{a\in A_j}a=0$ for $j\in\{1,2\}$ and $f(v^iv_i)\notin A_1$ (it is trivial for $p=2$, while otherwise: if $f(v^iv_i)\in I^*$ it follows directly by Lemma~\ref{involutions2}, and if $f(v^iv_i)=0$ then it is sufficient to apply Lemma~\ref{involutions} to obtain appropriate sets $A'_1,A'_2,A'_3$ with $|A'_1|=2^{p}-4, |A'_2|=1, |A'_3|=3$ and set $A_1=A'_1$, $A_2=A'_2\cup A'_3$). We take elements from $A_1$ and $(r_i+4-2^p)/2$ distinct pairs $(d_1,-d_1),$ $\ldots,$ $(d_{(r_i+4-2^p)/2},-d_{(r_i+4-2^p)/2})$ from the  elements of $\gr$  and  label all edges from $N_i$ with them.
Now if $r_i= 2^p-2$, then $k\geq 2^{p+1}$, and hence we may take
$r_i/2$ distinct pairs $(d_1,-d_1),$ $\ldots,$ $(d_{r_i/2},-d_{r_i/2})$ from the  elements of $\gr$ not including $f(v^iv_i)$ and  label all edges from $N_i$ with these. Finally, if $r_i\leq 2^p-4$, then similarly as above, by Lemma~\ref{involutions2} or~\ref{involutions}
there exists a partition $A=\{A_1, A_2\}$ of $I$ such that $|A_1|=r_i$,   $|A_2|=2^p-r_i$,  $\sum_{a\in A_j}a=0$ for $j\in\{1,2\}$ and $f(v^iv_i)\notin A_1$. We then label the edges in $N_i$ with all elements from $A_1$.
\end{proof}\\

Observe that if $T$ is a tree with  maximum degree $\Delta(T)=2^p-3\geq5$
such that  any vertex of degree $2$ has at least one neighbor of degree $2$ that is not isomorphic to a  $(2^p-3)$-regular tree then using exactly the same method as in the proof of Lemma~\ref{inwolucje} for the root  of degree $r\not\in \{1, 2^p-3\}$ we obtain that $T$ has a $(\zet_2)^p$-twin edge coloring.

By 
Theorem~\ref{drzewa} and
Lemmas~\ref{inwolucja} and \ref{inwolucje} we deduce the following.

\begin{myobservation} Let $T$ be a tree with   even maximum degree such that  any vertex of degree $2$ has at least one neighbor of degree $2$. If  $\Delta(T)\neq 2^p-2$ for every integer $p$, then $\chi'_g(T)\leq \Delta(T)+2$. \qed
\end{myobservation}

Finally, we obtain the following upper bound, which is tight according to Observation~\ref{sharp}.

\begin{mytheorem}
If $T$ is a tree of order at least $3$ then $\chi'_g(T)\leq \Delta(T)+2$ if $\Delta(T)$ is odd and $T$ is not $(3^{p}-2)$-regular for some integer $p\geq2$ and $\chi'_g(T)\leq \Delta(T)+3$ otherwise.\label{drzewaNEW}
\end{mytheorem}
\begin{proof}
Observe that for $k\in\{3,5,6,7\}$ all groups of order $k$ are isomorphic to $\zet_k$, and therefore every tree $T$ having maximum degree at most $5$ has $\chi'_g(T)\leq \Delta(T)+2$ by Theorem~\ref{drzewa}. Thus we may assume that $\Delta(T)\geq6$. For $T$ not being regular 
we are then done by Lemma~\ref{treeodd}. 
Assume now 
$T$ is a regular tree. If $\Delta(T)\neq 3^{p}-2$ for every integer $p$ then we are again done by Lemma~\ref{treeodd}. 
Suppose then that $\Delta(T)= 3^{p}-2$ and 
$\gr$ is a group of order 
$3^{p }+1$.
Then
$\Delta(T)+3$ is even but cannot be equal to $2^r$ for any natural number
$r$ by the Catalan-Mih\v{a}ilescu Theorem  \cite{Mih}, and therefore the existence of
a $\gr$-twin edge coloring of $T$ follows
by~Lemmas~\ref{inwolucja} and \ref{inwolucje}.
\end{proof}

\section{General upper bound}

Recall that for a given graph $G$ by ${\rm col}(G)$ we denote its coloring number, that is the least integer $k$ such that each subgraph of $G$ has minimum degree less than $k$.
Equivalently, it is the smallest $k$ for which we may linearly order all vertices of $G$ into a sequence $v_1,v_2,\ldots,v_n$ so that every vertex $v_i$ has at most $k-1$ neighbors preceding it in the sequence.
Hence ${\rm col}(G)\leq \Delta(G)+1$.
Note that ${\rm col}(G)$ equals the degeneracy of $G$ plus $1$, and thus the result below may 
be formulated 
in terms of either of the two graph invariants.
\begin{mytheorem}\label{UpperBoundTh}
If $G$ is a connected graph of order at least $3$ then $\chi'_g(G)\leq 2(\Delta(G)+{\rm col}(G))-5$.
\end{mytheorem}

\begin{proof}
Suppose first that ${\rm col}(G)=2$. For $\Delta(G)=2$, the statement of the theorem is true by Theorem~\ref{drzewa},
while for $\Delta(G)\geq 3$ it follows from Theorem~\ref{drzewaNEW}.

So we may assume that ${\rm col}(G)\geq 3$.
Fix any Abelian group $\gr$ of order $|\gr|\geq 2(\Delta(G)+{\rm col}(G))-5$.
Let $v_1,v_2,\ldots,v_n$ be the ordering of $V(G)$ witnessing the value of ${\rm col}(G)$.
We will 
label the edges of $G$ with elements of $\gr$ in $n-1$ stages, each corresponding to a consecutive vertex from among $v_2,v_3,\ldots,v_n$.
Initially no edge is labeled. 
Then at each stage $i$, $i=2,3,\ldots,n$, 
we label all \emph{backward edges} of $v_i$, i.e. every edge $v_jv_i\in E$ with $j<i$; such a vertex $v_j$ is called a \emph{backward neighbor} of $v_i$.
We will choose labels avoiding (most of the) sum conflicts between already analyzed vertices and so that at all times the partial edge coloring obtained is proper.
To this end we will make sure that at the end of every stage $i$, the conditions ($1^\circ$)-($3^\circ$) below hold. 
Let $I_i$ denote the set of indices $j$ of all vertices $v_j$ in $\{v_1,v_2,\ldots,v_i\}$ each of which has a neighbor $v_k$ of degree $1$ in $G$ with $k>i$ (note that for any $i$, if an index $j\in\{1,\ldots,i\}$ does not belong to $I_i$, then $j$ does not belong to any set $I_t$ with $t\geq i$). By $w(v_t)$, for each $t$ in $\{1,\ldots,n\}$, we mean the contemporary sum at $v_t$ (with unlabeled edges contributing $0$ to such a sum): 
\begin{itemize}
\item[($1^\circ$)] adjacent edges must be labeled differently;
\item[($2^\circ$)] for every 
$j\in I_i$ such that $v_j$ has a neighbor in $\{v_1,\ldots,v_i\}$:
$w(v_j)\neq 0$;
\item[($3^\circ$)]  for every edge $v_jv_k\in E(G)$ such that 
$j,k\notin I_i$, $j<k\leq i$
and 
$v_k$ has at least $2$ neighbors in $\{v_1,\ldots,v_i\}$ or $v_j$ has a neighbor in $\{v_{k+1},\ldots,v_i\}$: 
$w(v_j)\neq w(v_k)$.
\end{itemize}
Note that if we are able to assure ($1^\circ$)-($3^\circ$) to hold after every stage, then
the edge coloring of $G$ obtained  at the end of our construction will be proper (by ($1^\circ$)), and moreover the neighbors will be distinguished by their corresponding sums, as desired. To see the latter of these, i.e. that $w(v_j)\neq w(v_k)$ for every edge $v_jv_k\in E(G)$ with $1\leq j<k\leq n$, consider first the case when ${\rm deg}(v_k)\geq 2$. Then the fact that $w(v_j)\neq w(v_k)$ follows directly from ($3^\circ$), as  $I_i=\emptyset$ by definition for $i=n$. Assume thus that ${\rm deg}(v_k)=1$, and hence ${\rm deg}(v_j)\geq 2$. Denote by $v_t$ the neighbor of $v_j$ in $G$ with the largest index $t$ (hence $t\geq k$).
Then if $t>k$, by ($3^\circ$) we must have had $w(v_j)\neq w(v_k)$ after stage $t$ and this could not change in the further part of the construction (as no other unlabeled edges incident with $v_j$ or $v_k$ are left after stage $t$).
If finally $t=k$, by ($2^\circ$) we must have had $w(v_j)\neq 0=w(v_k)$ after stage $k-1$ and the inequality $w(v_j)\neq w(v_k)$ could not be violated regardless of the choice of the label for $v_jv_k$ in the following stage (nor in any further ones).
In order to prove the theorem it is thus indeed sufficient to show that we are able to satisfy ($1^\circ$)-($3^\circ$) after every stage using labels in $\gr$.

So assume we are about to perform step $i$ of the construction for some $i\in\{2,\ldots,n\}$ and thus far all our requirements have been fulfilled.
Let $v_{i_1},v_{i_2},\ldots,v_{i_b}$ be all backward neighbors of $v_i$ (hence $b\leq {\rm col}(G)-1$), and set $b=0$ if there are none. 
If $b>0$, subsequently for $j=1,2,\ldots,b-1$ (if $b>1$) we choose weights for $v_{i_j}v_i$ consistently with our rules.
Thus by ($1^\circ$) we cannot use at most $(\Delta(G)-1)+(j-1)\leq \Delta(G)+{\rm col}(G)-4$ labels of already colored edges adjacent with $v_{i_j}v_i$ to label it. If $i_j\in I_i$, in order to obey ($2^\circ$) we cannot use at most one more label for $v_{i_j}v_i$ so that $w(v_{i_j})\neq 0$ afterwards,
while if $i_j\notin I_i$ we will choose the label for $v_{i_j}v_i$ so that afterwards the sum at $v_{i_j}$ is distinct from the sums of all its neighbors except possibly $v_i$ -- this blocks at most $\Delta(G)-1$ additional labels for $v_{i_j}v_i$. As we have thus altogether at most
$\Delta(G)+{\rm col}(G)-4+\Delta(G)-1$ forbidden labels, we are left with at least ${\rm col}(G)\geq 3$ available options in $\gr$ to label $v_{i_j}v_i$.
We choose any of these, except for the case when $j=b-1$, when we make the choice so that $w(v_{i_b})\neq w(v_i)$ afterwards (note that though the choice of label for $v_{i_{b-1}}v_i$ is performed before the one for $v_{i_b}v_i$, the latter one counts in the the sums of both, $v_{i_b}$ and $v_i$, and thus the label of $v_{i_b}v_i$ does not influence the distinction of $w(v_{i_b})$ from $w(v_i)$).
Next, for $v_{i_b}v_i$ we analogously as above cannot use at most $\Delta(G)+{\rm col}(G)-3$ labels by ($1^\circ$).
Moreover, again at most $1$ additional label might be blocked for $v_{i_b}v_i$ by ($2^\circ$) if $i_b\in I_i$ or at most $\Delta(G)-1$ labels otherwise -- so that the obtained sum at $v_{i_b}$ is distinct from the sums of its neighbors except possibly $v_i$.
Similarly at most $1$ additional label might be blocked for $v_{i_b}v_i$ by ($2^\circ$) if $i\in I_i$ or at most ${\rm col}(G)-2$ labels otherwise -- so that the obtained sum at $v_i$ is distinct from the sums of its backward neighbors except possibly $v_{i_b}$.
Altogether at most $(\Delta(G)+{\rm col}(G)-3)+(\Delta(G)-1)+({\rm col}(G)-2)<|\gr|$ labels are thus forbidden, and hence we have at least one label available for $v_{i_b}v_i$.

By such a construction it is clear that conditions ($1^\circ$) and ($2^\circ$) hold after stage $i$.
As for condition ($3^\circ$) it is also straightforward that it holds
for every edge $v_jv_k\in E(G)$ such that $j,k\notin I_i$, $j<k\leq i$ except possibly $v_{i_b}v_i$, for which we were admitting a possible conflict $w(v_{i_b})=w(v_i)$. By our construction this however could only happen if $b=1$, as for $b\geq 2$ we prevented this by our choice of the label of $v_{i_{b-1}}v_i$. Then however $v_{i_b}v_i$ does not meet the assumptions of ($3^\circ$) after stage $i$, so we need not have $w(v_{i_b})\neq w(v_i)$ according to our rules.

As discussed earlier, after step $n$ of the construction we obtain a desired edge labeling of $G$.
\end{proof}

\section{Concluding remarks}\label{SectionOurConcludingRemarks}

By Theorem~\ref{UpperBoundTh} we in particular obtain that the following is true.
\begin{mycorollary}
If $G$ is a connected planar graph of order at least $3$ then $\chi'_g(G)\leq 2\Delta(G)+7$.
\end{mycorollary}
We believe however that a stronger upper bound should hold even for all graphs,
and we pose 
the following conjecture.

\begin{myconjecture}
If $G$ is a connected graph of order at least $3$ then $\chi'_g(G)\leq \Delta(G)+3$.
\end{myconjecture}
By Observation~\ref{sharp} this could not be improved.
In this context it would also be interesting to settle for which trees $T$ we actually have the equality $\chi'_g(T) = \Delta(T)+3$.

At the end  notify that so far  only for only one family of trees (namely    $(3^{p}-2)$-regular trees with $p\geq2$) we can conclude  that there exists a $\gr$-twin edge coloring for any Abelian group $\gr$ of order $k\geq \chi'_g(G)$. Recall that the fact that $\chi'_g(G)\leq K$ for a given graph $G$ and a constant $K$ does not guarantee that for every group $\gr$ of order $k>K$ there exists a $\gr$-twin edge coloring  of  $G$.  We may  almost guarantee this in the case of trees though. Note that for an Abelian group $\gr$, $|\gr|\geq10$,  with at most one involution  we can improve Lemma~\ref{trzyzero} so that $x,y$ are additionally nonzero. Moreover for $k\geq3$ and $n_1\geq3$ in Lemma~\ref{involutions2} we can require that also $0\notin A_1$. Hence using a similar method as in the proofs of Lemmas~\ref{treeodd}, \ref{inwolucja} and~\ref{inwolucje} one can show that for any Abelian group $\gr$ of order $k>\Delta(T)+3$. 
With a bit of extra effort we thus could obtain upper bounds of similar flavor as the ones in Theorem~\ref{drzewaNEW} for all forests.
These we can however derive effortlessly from our results for the case of labelings with cyclic groups $\zet_k$.
Thus we conclude our paper 
with Theorem~\ref{ConcludingMytheorem} below, containing exactly these upper bounds 
for the twin chromatic index of forests.

One can easily see that a path of order at least $3$ has a $\zet_k$-twin edge labeling for any odd $k\geq3$. Then by Theorem~\ref{drzewa} and Lemma~\ref{treeodd}  we directly obtain the following.
\begin{mytheorem}\label{ConcludingMytheorem}
Let $G$ be a forest containing no isolated edges. Then $\chi'_t(G)\leq \Delta(G) +2$ if $\Delta(G)$ is odd and $\chi'_t(G)\leq \Delta(G) +3$ otherwise.
\end{mytheorem}

\nocite{*}
\bibliographystyle{alpha}
\bibliography{twin_dmtcs}

\newcommand{\etalchar}[1]{$^{#1}$}
\begin{thebibliography}{ABDM{\etalchar{+}}07}

\bibitem[ABDM{\etalchar{+}}07]{ref_AddDalMcDReeTho}
L.~Addario-Berry, K.~Dalal, C.~McDiarmid, B.A. Reed, and A.~Thomason.
\newblock Vertex-colouring edge-weightings.
\newblock {\em Combinatorica}, 27(1):1--12, 2007.

\bibitem[ABDR08]{ref_AddDalRee}
L.~Addario-Berry, K.~Dalal, and B.A. Reed.
\newblock Degree constrained subgraphs.
\newblock {\em Discrete Applied Mathematics}, 156(7):1168 -- 1174, 2008.
\newblock GRACO 2005.

\bibitem[AC16]{ref_AnhCic1}
M.~Anholcer and S.~Cichacz.
\newblock Group sum chromatic number of graphs.
\newblock {\em European Journal of Combinatorics}, 55:73 -- 81, 2016.

\bibitem[AHJ{\etalchar{+}}14]{ref_AndHekJohVerPin}
E.~Andrews, L.~Helenius, D.~Johnston, J.~VerWys, and P.~Zhang.
\newblock On twin edge colorings of graphs.
\newblock {\em Discussiones Mathematicae Graph Theory}, 34(3):613--627, 2014.

\bibitem[AJZ14]{ref_AndJohPin}
E.~Andrews, D.~Johnston, and P.~Zhang.
\newblock A twin edge coloring conjecture.
\newblock {\em Bulletin of the Institute of Combinatorics and its
  Applications}, 70:28--44, 2014.

\bibitem[AT90]{ref_AigTri}
M.~Aigner and E.~Triesch.
\newblock Irregular assignments of trees and forests.
\newblock {\em SIAM Journal on Discrete Mathematics}, 3(4):439--449, 1990.

\bibitem[AT98]{ref_AmaTog}
D.~Amar and O.~Togni.
\newblock Irregularity strength of trees.
\newblock {\em Discrete Mathematics}, 190(1):15 -- 38, 1998.

\bibitem[BP17]{BonamyPrzybylo}
M.~Bonamy and J.~Przyby{\l}o.
\newblock On the neighbor sum distinguishing index of planar graphs.
\newblock {\em Journal of Graph Theory}, 85(3):669--690, 2017.

\bibitem[Cic17]{ref_Cic}
S.~Cichacz.
\newblock On zero sum-partition of abelian groups into three sets and group
  distance magic labeling.
\newblock {\em Ars Mathematica Contemporanea}, 13(2):417--425, 2017.

\bibitem[CJL{\etalchar{+}}88]{ref_ChaJacLehOelRuiSab1}
G.~Chartrand, M.S. Jacobson, J.~Lehel, O.R. Oellermann, S.~Ruiz, and F.~Saba.
\newblock Irregular networks.
\newblock {\em Congressus Numerantium}, 64:187--192, 1988.

\bibitem[CL08]{Lazebnik}
B.~Cuckler and F.~Lazebnik.
\newblock Irregularity strength of dense graphs.
\newblock {\em Journal of Graph Theory}, 58(4):299--313, 2008.

\bibitem[CNP04]{CN}
D.~Combe, A.M. Nelson, and W.D. Palmer.
\newblock Magic labellings of graphs over finite abelian group.
\newblock {\em The Australasian Journal of Combinatorics}, 29:259--271, 2004.

\bibitem[DWZ14]{DongWang_mad}
A.~Dong, G.~Wang, and J.~Zhang.
\newblock Neighbor sum distinguishing edge colorings of graphs with bounded
  maximum average degree.
\newblock {\em Discrete Applied Mathematics}, 166:84 -- 90, 2014.

\bibitem[FGKP02]{Frieze}
A.~Frieze, R.J. Gould, M.~Karo\'nski, and F.~Pfender.
\newblock On graph irregularity strength.
\newblock {\em Journal of Graph Theory}, 41(2):120--137, 2002.

\bibitem[FMP{\etalchar{+}}13]{FlandrinMPSW}
E.~Flandrin, A.~Marczyk, J.~Przyby{\l}o, J.-F. Sacl\'e, and M.~Wo{\'z}niak.
\newblock Neighbor sum distinguishing index.
\newblock {\em Graphs and Combinatorics}, 29(5):1329--1336, 2013.

\bibitem[FSJL89]{Faudree2}
R.J. Faudree, R.H. Schelp, M.S. Jacobson, and J.~Lehel.
\newblock Irregular networks, regular graphs and integer matrices with distinct
  row and column sums.
\newblock {\em Discrete Mathematics}, 76(3):223 -- 240, 1989.

\bibitem[Gal09]{ref_Gal}
J.~Gallian.
\newblock {\em Contemporary Abstract Algebra}.
\newblock Cengage Learning, 2009.

\bibitem[Hat05]{Hatami}
H.~Hatami.
\newblock $\triangle$+300 is a bound on the adjacent vertex distinguishing edge
  chromatic number.
\newblock {\em Journal of Combinatorial Theory, Series B}, 95(2):246 -- 256,
  2005.

\bibitem[Joh15]{ref_Jon}
D.~Johnston.
\newblock {\em Edge colorings of graphs and their applications}.
\newblock PhD thesis, Western Michigan University, USA, 2015.

\bibitem[KKP10]{ref_KalKarPfe2}
M.~Kalkowski, M.~Karoński, and F.~Pfender.
\newblock Vertex-coloring edge-weightings: Towards the 1-2-3-conjecture.
\newblock {\em Journal of Combinatorial Theory, Series B}, 100(3):347 -- 349,
  2010.

\bibitem[KKP11]{ref_KalKarPfe1}
M.~Kalkowski, M.~Karoński, and F.~Pfender.
\newblock A new upper bound {for~the~irregularity} strength of graphs.
\newblock {\em SIAM Journal on Discrete Mathematics}, 25:1319--1321, 2011.

\bibitem[K{\L}T04]{ref_KarLucTho}
M.~Karoński, T.~{\L}uczak, and A.~Thomason.
\newblock Edge weights and vertex colours.
\newblock {\em Journal of Combinatorial Theory, Series B}, 91(1):151 -- 157,
  2004.

\bibitem[Leh91]{ref_Leh}
J.~Lehel.
\newblock Facts and quests on degree irregular assignments.
\newblock {\em Graph Theory, Combinatorics and Applications}, 2:765--782, 1991.

\bibitem[Mih04]{Mih}
P.~Mih\v{a}ilescu.
\newblock Primary cyclotomic units and a proof of catalan’s conjecture.
\newblock {\em Journal f\"ur die reine und angewandte Mathematik},
  572:167--195, 2004.

\bibitem[MP14]{ref_Prz3}
P.~Majerski and J.~Przyby{\l}o.
\newblock On the irregularity strength of dense graphs.
\newblock {\em SIAM Journal on Discrete Mathematics}, 28:197--205, 2014.

\bibitem[Nie00]{Nierhoff}
T.~Nierhoff.
\newblock A tight bound on the irregularity strength of graphs.
\newblock {\em SIAM Journal on Discrete Mathematics}, 13(3):313--323, 2000.

\bibitem[Prz13]{Przybylo_CN_1}
J.~Przyby{\l}o.
\newblock Neighbor distinguishing edge colorings via the combinatorial
  nullstellensatz.
\newblock {\em SIAM Journal on Discrete Mathematics}, 27(3):1313--1322, 2013.

\bibitem[PW15]{Przybylo_CN_2}
J.~Przyby{\l}o and T.-L. Wong.
\newblock Neighbor distinguishing edge colorings via the combinatorial
  nullstellensatz revisited.
\newblock {\em Journal of Graph Theory}, 80(4):299--312, 2015.

\bibitem[TWZ16]{ref_ThoWuZha}
C.~Thomassen, Y.~C.-Q. Wu, and Z.~Zhang.
\newblock The $3$-flow conjecture, factors modulo k, and the
  $1-2-3$-conjecture.
\newblock {\em Journal of Combinatorial Theory, Series B}, 121:308--325, 2016.

\bibitem[WCW14]{WangChenWang_planar}
G.~Wang, Z.~Chen, and J.~Wang.
\newblock Neighbor sum distinguishing index of planar graphs.
\newblock {\em Discrete Mathematics}, 334:70 -- 73, 2014.

\bibitem[WY08]{ref_WanYu}
T.~Wang and Q.~Yu.
\newblock On vertex-coloring 13-edge-weighting.
\newblock {\em Frontiers of Mathematics in China}, 3(4):581--587, 2008.

\bibitem[ZLW02]{Zhang}
Z.~Zhang, L.~Liu, and J.~Wang.
\newblock Adjacent strong edge coloring of graphs.
\newblock {\em Applied Mathematics Letters}, 15(5):623 -- 626, 2002.

\end{thebibliography}
\label{sec:biblio}

\end{document}